\documentclass[11pt,a4paper]{article}
\usepackage[T1]{fontenc}
\usepackage[utf8]{inputenc}
\usepackage{lmodern}
\usepackage[margin=27mm]{geometry}
\usepackage{amsmath,amssymb,amsthm,mathtools}
\usepackage{booktabs,array,tabularx,longtable}
\usepackage[numbers,sort&compress]{natbib}
\usepackage{algorithm,algpseudocode}
\usepackage{microtype}
\usepackage{xcolor}
\usepackage[hidelinks]{hyperref}
\usepackage[nameinlink,noabbrev]{cleveref}
\usepackage{enumitem}
\renewcommand{\arraystretch}{1.16}
\numberwithin{equation}{section}
\newtheorem{theorem}{Theorem}[section]
\newtheorem{lemma}[theorem]{Lemma}
\newtheorem{proposition}[theorem]{Proposition}

\theoremstyle{definition}
\newtheorem{definition}[theorem]{Definition}
\theoremstyle{remark}

\DeclareMathOperator{\lcm}{lcm}

\newcommand{\one}{\mathbf{1}}

\newcommand{\pending}[1]{\par\noindent\textit{Data to be supplied. #1}\par}

\newcommand{\ReferenceSnapshot}{September 25, 2026}

\newcommand{\TotalAtFifty}{4567863967927930120629311714346947864974124915610296}
\newcommand{\ConnectedAtFifty}{4567863812490955416720841969219753968607349955976192}
\newcommand{\DisconnectedAtFifty}{155436974703908469745127193896366774959634104}
\newcommand{\UnrestrictedTableRows}{%
29 & \texttt{10652577590214241002618} \\
30 & \texttt{207054244547337985329720} \\
31 & \texttt{4161993706714940326838786} \\
32 & \texttt{86424005407555318218067697} \\
33 & \texttt{1851998440392256616901456520} \\
34 & \texttt{40917079037806933457347198438} \\
35 & \texttt{931183927432486423640202376814} \\
36 & \texttt{21810450084104345091773820858196} \\
37 & \texttt{525344362944702962413733395520252} \\
38 & \texttt{13003032308773638146164487141818952} \\
39 & \texttt{330487084031595634132232172180021294} \\
40 & \texttt{8619415114072964191751466252612719377} \\
41 & \texttt{230533227247675454285548999068702309794} \\
42 & \texttt{6319087592497085783571337132123210802321} \\
43 & \texttt{177413127187497235740886087991800080245168} \\
44 & \texttt{5098998905389251773034803769266361673511731} \\
45 & \texttt{149940921587171855739584362970984343156657193} \\
46 & \texttt{4508894979370373105831885412360496862130843299} \\
47 & \texttt{138587154167033251155034541552679549848762484979} \\
48 & \texttt{4351870731024217447549685703640756912833949879605} \\
49 & \texttt{139551606359121154554275644335312475951018945402502} \\
50 & \texttt{4567863967927930120629311714346947864974124915610296} \\
}
\newcommand{\ConnectedTableRows}{%
29 & \texttt{10652570446223068758328} \\
30 & \texttt{207054129798202478947780} \\
31 & \texttt{4161991787056364392993657} \\
32 & \texttt{86423972013842831141067779} \\
33 & \texttt{1851997837239534197240070489} \\
34 & \texttt{40917067742027646230636613788} \\
35 & \texttt{931183708364518646087601058142} \\
36 & \texttt{21810445689689214773057789501098} \\
37 & \texttt{525344271868938707705400377687514} \\
38 & \texttt{13003030360567200772187592732217381} \\
39 & \texttt{330487041060514755173477290891267539} \\
40 & \texttt{8619414137666184155650741761970634546} \\
41 & \texttt{230533204411383026272756189104337347077} \\
42 & \texttt{6319087043198021512104740092530777402157} \\
43 & \texttt{177413113609115042486441572202257674562253} \\
44 & \texttt{5098998560699564475515004308783887431914087} \\
45 & \texttt{149940912607692268107701695024603153774329415} \\
46 & \texttt{4508894739468701917064672399081911536531357642} \\
47 & \texttt{138587147597943684444165814756297043282778275999} \\
48 & \texttt{4351870546772975644217239817541796928811187365690} \\
49 & \texttt{139551601068525711134953857116015783170154234904858} \\
50 & \texttt{4567863812490955416720841969219753968607349955976192} \\
}
\newcommand{\DisconnectedTableRows}{%
29 & \texttt{7143991172244290} \\
30 & \texttt{114749135506381940} \\
31 & \texttt{1919658575933845129} \\
32 & \texttt{33393712487076999918} \\
33 & \texttt{603152722419661386031} \\
34 & \texttt{11295779287226710584650} \\
35 & \texttt{219067967777552601318672} \\
36 & \texttt{4394415130318716031357098} \\
37 & \texttt{91075764254708333017832738} \\
38 & \texttt{1948206437373976894409601571} \\
39 & \texttt{42971080878958754881288753755} \\
40 & \texttt{976406780036100724490642084831} \\
41 & \texttt{22836292428012792809964364962717} \\
42 & \texttt{549299064271466597039592433400164} \\
43 & \texttt{13578382193254444515789542405682915} \\
44 & \texttt{344689687297519799460482474241597644} \\
45 & \texttt{8979479587631882667946381189382327778} \\
46 & \texttt{239901671188767213013278585325599485657} \\
47 & \texttt{6569089566710868726796382506565984208980} \\
48 & \texttt{184251241803332445886098959984022762513915} \\
49 & \texttt{5290595443419321787219296692780864710497644} \\
50 & \texttt{155436974703908469745127193896366774959634104} \\
}
\newcommand{\LabeledTableRows}{%
29 & \begin{tabular}[t]{@{}l@{}}\texttt{92806550484635529288139898051471032653850237781032275}\end{tabular} \\
30 & \begin{tabular}[t]{@{}l@{}}\texttt{54178571689579205256256214619696958221103299101023824201}\end{tabular} \\
31 & \begin{tabular}[t]{@{}l@{}}\texttt{33794692140133684291976181563114413932853044724371478285965}\end{tabular} \\
32 & \begin{tabular}[t]{@{}l@{}}\texttt{224763286005461730802396972091242969872466891338678537848146}\\\texttt{65}\end{tabular} \\
33 & \begin{tabular}[t]{@{}l@{}}\texttt{159073630682530822773001994871084089035561514299587586469104}\\\texttt{94880}\end{tabular} \\
34 & \begin{tabular}[t]{@{}l@{}}\texttt{119579877057814709387685985273327811803694860047609448813608}\\\texttt{20696280}\end{tabular} \\
35 & \begin{tabular}[t]{@{}l@{}}\texttt{953109923170547503360500265508469672579523176079841352310514}\\\texttt{5502142456}\end{tabular} \\
36 & \begin{tabular}[t]{@{}l@{}}\texttt{804143556779821413735872206160170572247089632029450508360488}\\\texttt{3838990780280}\end{tabular} \\
37 & \begin{tabular}[t]{@{}l@{}}\texttt{717050864044352640732976683055676392704654911582065784578058}\\\texttt{3858385663594705}\end{tabular} \\
38 & \begin{tabular}[t]{@{}l@{}}\texttt{674758589928248584062426523344333443900935569857323195900301}\\\texttt{5073131753785087135}\end{tabular} \\
39 & \begin{tabular}[t]{@{}l@{}}\texttt{669142607035471282503067807638506540350652512913663389342784}\\\texttt{1539740916073979744185}\end{tabular} \\
40 & \begin{tabular}[t]{@{}l@{}}\texttt{698363239215059738245539752094272830050776679909764349632829}\\\texttt{1058410000012155962088891}\end{tabular} \\
41 & \begin{tabular}[t]{@{}l@{}}\texttt{766100463705513418289509463348743309395364027639323880461859}\\\texttt{0583829648033219052642534620}\end{tabular} \\
42 & \begin{tabular}[t]{@{}l@{}}\texttt{882282273153181050587842980294934328474074911127483823523757}\\\texttt{0532338773478370886237455063670}\end{tabular} \\
43 & \begin{tabular}[t]{@{}l@{}}\texttt{106548503924788074363035230440413823742057428889485399120486}\\\texttt{56013463742155160211803630489394990}\end{tabular} \\
44 & \begin{tabular}[t]{@{}l@{}}\texttt{134780986530494102461477716708860728684304231128344855401710}\\\texttt{92277281668055019288737379454463444690}\end{tabular} \\
45 & \begin{tabular}[t]{@{}l@{}}\texttt{178400606655402808635027847452596270382591797116985366991285}\\\texttt{36863129266008073886119166324430132289281}\end{tabular} \\
46 & \begin{tabular}[t]{@{}l@{}}\texttt{246839940426317257740990247401343710635834173714001805047760}\\\texttt{41882750926291040060492349176190339948158835}\end{tabular} \\
47 & \begin{tabular}[t]{@{}l@{}}\texttt{356672550418434757956415106780554116249786756963236525829491}\\\texttt{06998987632002300689390223954125671277135249035}\end{tabular} \\
48 & \begin{tabular}[t]{@{}l@{}}\texttt{537725206304142856901413691985529301431754133732733524131297}\\\texttt{47301758107575592331280641380275202281129189032395}\end{tabular} \\
49 & \begin{tabular}[t]{@{}l@{}}\texttt{845094630156391443416560961473494251019655971871030312651365}\\\texttt{20290279158256553149485563198292823829141082013713120}\end{tabular} \\
50 & \begin{tabular}[t]{@{}l@{}}\texttt{138336823347162818108015410281773931110899524325911679455625}\\\texttt{070087813910376506171032162681326294383945945551868904476}\end{tabular} \\
}

\title{Exact counting of unlabeled quartic graphs by permutation-cycle aggregation}
\author{Yue Cheng and Zhipeng Xu\thanks{Corresponding author: xuzhp@ntu.edu.cn}\\[0.4em]
  \small School of Mathematics and Statistics, Nantong University,\\
  \small Nantong, China\\
}
\date{}

\begin{document}
\maketitle

\begin{abstract}
The number of unlabeled regular graphs can be expressed as an average of fixed-point counts over vertex permutations, but evaluating each fixed-point count still requires the degree constraints to be enforced. We give an exact recurrence that processes one complete permutation cycle at a time and records the remaining cycles only by their lengths and residual degrees. The recurrence combines internal edge orbits with orbits joining distinct cycles, while binomial and multinomial coefficients retain the multiplicities of choices that lead to the same remaining state. We prove that this state description is sufficient under complete-cycle elimination and derive bounds on the number of states and transitions. For every fixed degree, the resulting algorithm has an $\exp(O(\sqrt n))$ upper bound in the number of vertices, including integer-arithmetic costs. The quartic case requires only four positive residual-degree classes for each cycle length. Small-instance comparisons with a separately implemented, vertex-indexed edge-orbit calculation verify both regular and nonuniform residual-degree inputs. The quartic calculation gives unrestricted and connected counts through order 50, including 22 orders beyond the corresponding reference tables through order 28. Connected counts are recovered by the inverse Euler transform, and all 22 identity-permutation contributions for orders 29--50 agree with the published labeled counts. The nonidentity fixed-point terms at these orders have not been independently recomputed.
\end{abstract}

\noindent\textbf{Keywords}\quad
regular graphs; exact counting; Burnside's lemma; permutation cycles; dynamic programming

\section{Introduction}
\label{sec:introduction}

Counting regular graphs up to isomorphism is different from constructing a representative of every isomorphism class. A generation algorithm must distinguish the graphs that it produces, whereas a counting algorithm may combine many choices whenever their remaining counting problems are identical. This distinction is particularly relevant when the required output is an integer table rather than a collection of adjacency matrices, because the number of graphs need not determine the number of states visited by an exact counting algorithm.

Regular graph enumeration belongs to the classical study of locally restricted graphs \citep{Read1959,Read1960}, and the use of permutation-group actions in unlabeled graph counting is developed systematically in \citet{HararyPalmer1973}. Constructive methods remain necessary when graph representatives or additional structural properties are required. Meringer's regular-graph generator uses orderly generation and canonicity criteria \citep{Meringer1999}, and parallel regular-graph enumeration has also been studied by \citet{Xu2019}. These approaches address a different output requirement from the count-only calculation considered here.

For labeled graphs, aggregation by degree already removes much of the redundancy in a direct search. Howroyd's publicly available program accumulates labeled graph counts by degree histograms, represents those histograms by polynomials, and uses binomial coefficients when selecting vertices from a degree class \citep{Howroyd2019}. Recursive counting for a prescribed labeled degree sequence is also considered by \citet{Kaygun2021}. Thus neither degree aggregation nor the reuse of repeated residual problems is, by itself, the distinction of the present method.

For an unlabeled count obtained from Burnside's lemma, the remaining problem is to count graphs fixed by a specified vertex permutation. Vertices on the same permutation cycle must have equal degrees in every invariant graph, and the available edges occur in entire orbits. A degree histogram alone does not describe these orbit constraints. We therefore aggregate complete permutation cycles jointly by cycle length and residual degree, and derive the transition weights from the edge orbits incident with one selected cycle.

The main result is an exact recurrence on these aggregated states, together with a proof of the condition under which the aggregation is valid. Complete-cycle elimination leaves every edge between the remaining active cycles undecided, so the number of completions depends only on the cycle-length and residual-degree histogram. We also bound the number of possible histograms by a colored-partition generating function and show that, for fixed degree, each state admits only polynomially many aggregated transitions. The analysis applies to any fixed degree, while the completed numerical tables concern quartic graphs of orders 29--50. These tables contain unrestricted and connected unlabeled counts, their disconnected differences, and the identity-permutation contributions, with exact decimal integers retained throughout.

The distinction from prior computational work must be stated at the level supported by the available sources. The OEIS entries for unrestricted and connected quartic graphs credit Howroyd with the corrected and extended values through order 28 \citep{OEISA033301,OEISA006820}, but those entries do not specify the complete algorithm that produced his unlabeled counts. The present comparison concerns his cited public labeled-graph program; it does not establish priority over an unpublished fixed-point implementation or a measured speed advantage over that implementation.

\section{Fixed graphs under a vertex permutation}
\label{sec:orbits}

\subsection{Burnside's formula}

All graphs in this paper are finite, simple, and undirected. Let $U_d(n)$ be the number of isomorphism classes of $d$-regular graphs on $n$ vertices, without a connectivity restriction, and set $U_d(0)=1$. For $n>0$, the count is zero when $d\geq n$ or $nd$ is odd. The symmetric group $S_n$ acts on labeled graphs by permuting their vertex labels.

Write a partition of $n$ as $\lambda=(1^{m_1}2^{m_2}\cdots n^{m_n})$ and define
\begin{equation}
 z_\lambda=\prod_{a=1}^n a^{m_a}m_a!.
 \label{eq:zlambda}
\end{equation}
If $F_d(\lambda)$ is the number of labeled $d$-regular graphs fixed by one permutation of cycle type $\lambda$, then the conjugacy-class form of Burnside's lemma gives \citep{HararyPalmer1973}
\begin{equation}
 U_d(n)=\sum_{\lambda\vdash n}\frac{F_d(\lambda)}{z_\lambda}.
 \label{eq:burnside}
\end{equation}
Although the terms of this sum are rational, the fixed-point counts themselves are nonnegative integers. An integer-only final accumulation is obtained from
\begin{equation}
 B_d(n)=\sum_{\lambda\vdash n}\frac{n!}{z_\lambda}F_d(\lambda),
 \qquad U_d(n)=\frac{B_d(n)}{n!}.
 \label{eq:integer-burnside}
\end{equation}

The quantity $F_d(\lambda)$ counts labeled edge sets, not edge sets modulo the centralizer of the chosen permutation. In particular, different choices of edge orbits remain distinct even when a rotation or exchange of equal-length cycles maps one choice to another. The averaging that removes vertex labels is performed only in \cref{eq:burnside}; no additional division by cycle rotations or cycle permutations is used inside a fixed-point calculation.

\subsection{Edge orbits within and between cycles}

Fix a permutation $\pi$ whose cycles have lengths $a_1,\ldots,a_t$. Its action on unordered pairs of distinct vertices partitions the possible edges into orbits, and a graph is fixed by $\pi$ precisely when its edge set is a union of those orbits. The degree contribution of an orbit is constant on each vertex cycle, which allows one variable per cycle to record the per-vertex degree contribution.

For a cycle of length $a$, an internal edge joins two positions at cyclic distance $j$ with $1\leq j\leq\lfloor a/2\rfloor$. Each distance $j<a/2$ gives one orbit contributing degree two at every vertex. If $a$ is even, distance $a/2$ instead gives a matching orbit contributing degree one. It follows that the polynomial for internal choices is
\begin{equation}
 P_a(x)=(1+x^2)^{\lfloor(a-1)/2\rfloor}(1+x)^{\one_{2\mid a}},
 \qquad I_a(u)=[x^u]P_a(x).
 \label{eq:internal}
\end{equation}
The conventions include $P_1(x)=1$ and $P_2(x)=1+x$, so fixed vertices introduce no loops and 2-cycles are treated without an exceptional rule in the recurrence.

\begin{lemma}
\label{lem:cross}
For two distinct vertex cycles of lengths $a$ and $b$, let $g=\gcd(a,b)$. Their $ab$ possible joining edges form $g$ orbits of length $\lcm(a,b)$. Every such orbit contributes $b/g$ edges at each vertex of the $a$-cycle and $a/g$ edges at each vertex of the $b$-cycle.
\end{lemma}
\begin{proof}
Index the cycles by $\mathbb{Z}/a\mathbb{Z}$ and $\mathbb{Z}/b\mathbb{Z}$. The action on a joining edge is the simultaneous shift $(i,j)\mapsto(i+1,j+1)$, whose orbit length is $\lcm(a,b)$. Dividing $ab$ by this length gives $g$ orbits. Within one orbit, each position in the first cycle is visited $\lcm(a,b)/a=b/g$ times, and each position in the second is visited $a/g$ times.
\end{proof}

Selecting $k$ of these $g$ orbits can therefore be done in $\binom{g}{k}$ ways and contributes per-vertex degrees $kb/g$ and $ka/g$. With $g_{ij}=\gcd(a_i,a_j)$, the complete fixed-point count has the coefficient representation
\begin{equation}
 F_d(\lambda)=
 [x_1^d\cdots x_t^d]
 \left\{\prod_{i=1}^t P_{a_i}(x_i)
 \prod_{1\leq i<j\leq t}
 \left(1+x_i^{a_j/g_{ij}}x_j^{a_i/g_{ij}}\right)^{g_{ij}}\right\}.
 \label{eq:cycle-product}
\end{equation}
This expression gives a direct description of the count, but the recurrence below avoids maintaining a separate variable and residual degree for every individual cycle.

\section{The remaining-state description}
\label{sec:states}

The elimination procedure chooses one complete vertex cycle, decides every still-undecided edge orbit incident with that cycle, and then removes the cycle. Consequently, at the boundary between two elimination steps, all edge orbits entirely within the remaining active vertices are undecided. Earlier choices affect the remaining problem only through the degree already supplied to each remaining vertex.

Because earlier choices were unions of edge orbits, all vertices of a remaining cycle have the same residual degree. A cycle with residual degree zero can be removed immediately, since every still-undecided orbit incident with it is forced absent. This removal contributes a factor of one and is applied before states are compared or stored.

\begin{definition}
A normalized state is a finitely supported array
\begin{equation}
 H=(H_{a,r})_{a\geq1,\,1\leq r\leq d},
 \label{eq:histogram}
\end{equation}
where $H_{a,r}$ is the number of active permutation cycles of length $a$ whose vertices each require $r$ further incident edges. Let $\Phi(H)$ denote the number of invariant edge sets on one fixed labeled realization of those cycles that supply exactly these residual degrees. Define
\begin{equation}
 N(H)=\sum_{a,r}aH_{a,r},
 \qquad R(H)=\sum_{a,r}arH_{a,r},
 \qquad \Phi(0)=1.
 \label{eq:state-size}
\end{equation}
\end{definition}

\begin{theorem}[Sufficiency of the histogram]
\label{thm:sufficient}
Suppose that two remaining problems satisfy the complete-cycle elimination condition and have the same normalized histogram $H$. Their numbers of invariant completions are equal. Thus $\Phi(H)$ does not depend on the identities of the active cycles, on their chosen starting positions, or on the history of previously eliminated cycles.
\end{theorem}
\begin{proof}
Match the active cycles of the two problems by their common length and residual degree. On each matched pair, choose any bijection that maps successive positions to successive positions. The union of these bijections conjugates the two restricted permutations and preserves every required residual degree.

Every unordered pair of distinct active vertices is still available in each problem, because no orbit entirely among active cycles has been decided. The bijection therefore maps edge orbits to edge orbits and sends every invariant completion of the first problem to an invariant completion of the second. Its inverse gives the reverse correspondence, proving equality of the counts. There is no need for the bijection to preserve edges incident with already eliminated vertices, since those edges affect a completion only through the residual degrees already matched.
\end{proof}

The theorem is a statement about the number of completions of one remaining labeled problem. It does not identify different labeled completions as the same graph, and it does not discard the multiplicity of earlier choices. When several earlier choices reach the same histogram, their weights multiply the same value $\Phi(H)$ and are added in the recurrence.

For a permutation of cycle type $\lambda$, the initial state is
\begin{equation}
 H^{\lambda,d}_{a,r}=m_a\one_{r=d},
 \qquad F_d(\lambda)=\Phi(H^{\lambda,d}),
 \label{eq:initial}
\end{equation}
with the zero-degree case handled by normalization. The elimination condition is essential. If some edges between active cycles have already been included, excluded, or otherwise restricted, two problems with the same histogram can have different sets of completions; an edge-by-edge procedure cannot generally use $H$ as its only state.

\section{An exact aggregated recurrence}
\label{sec:recurrence}

\subsection{Local choices and their multiplicities}

Let $H\neq0$ and choose a pair $(a,r)$ with $H_{a,r}>0$. One particular labeled cycle in that class is designated as the pivot, while its identity is omitted from the stored state. Remove that cycle from the histogram and write
\begin{equation}
 h_{b,s}=H_{b,s}-\one_{(b,s)=(a,r)}.
 \label{eq:pivot-removal}
\end{equation}
The pivot can be chosen by any deterministic rule depending only on $H$; choosing the lexicographically largest occupied pair $(a,r)$ is one such rule.

For a target class $(b,s)$, define
\begin{equation}
 g_{ab}=\gcd(a,b),\qquad
 \alpha_{ab}=\frac{b}{g_{ab}},\qquad
 \beta_{ab}=\frac{a}{g_{ab}},
 \label{eq:increments}
\end{equation}
and
\begin{equation}
 K_{a,r;b,s}=\min\left\{g_{ab},
 \left\lfloor\frac{r}{\alpha_{ab}}\right\rfloor,
 \left\lfloor\frac{s}{\beta_{ab}}\right\rfloor\right\}.
 \label{eq:kmax}
\end{equation}
If $k$ joining orbits are selected between the pivot and one target cycle, they consume $k\alpha_{ab}$ units of the pivot's per-vertex residual degree and $k\beta_{ab}$ units at each target vertex. The upper bound in \cref{eq:kmax} excludes degree overflow before an aggregated transition is formed.

Let $\nu_{b,s,k}$ be the number of target cycles in class $(b,s)$ that receive exactly $k$ joining orbits from the pivot. For each occupied target class, these nonnegative integers satisfy
\begin{equation}
 \sum_{k=0}^{K_{a,r;b,s}}\nu_{b,s,k}=h_{b,s}.
 \label{eq:class-conservation}
\end{equation}
The number of labeled choices represented by this allocation is
\begin{equation}
 W_{b,s}(\nu)=
 \frac{h_{b,s}!}{\prod_{k=0}^{K_{a,r;b,s}}\nu_{b,s,k}!}
 \prod_{k=0}^{K_{a,r;b,s}}
 \binom{g_{ab}}{k}^{\nu_{b,s,k}}.
 \label{eq:class-weight}
\end{equation}
The multinomial factor selects which labeled target cycles receive each orbit count, while the binomial factors select the actual joining orbits for those cycles.

\subsection{The recurrence and its proof}

Suppose that the pivot receives internal degree $u$, which has multiplicity $I_a(u)$ from \cref{eq:internal}. Let $\mathcal{A}(H,a,r,u)$ be the set of allocations satisfying \cref{eq:class-conservation} and the pivot-degree equation
\begin{equation}
 u+\sum_{b,s}\sum_{k=0}^{K_{a,r;b,s}}
 k\alpha_{ab}\nu_{b,s,k}=r.
 \label{eq:pivot-budget}
\end{equation}
The new normalized histogram is obtained by moving each target cycle to its resulting residual-degree class,
\begin{equation}
 H^{\nu}_{b,t}=
 \sum_{s=1}^d\sum_{k=0}^{K_{a,r;b,s}}
 \nu_{b,s,k}\one_{t=s-k\beta_{ab}},
 \qquad 1\leq t\leq d.
 \label{eq:next-state}
\end{equation}
Terms with $s-k\beta_{ab}=0$ do not appear in $H^\nu$, since those cycles are exhausted and removed.

\begin{theorem}[Aggregated fixed-point recurrence]
\label{thm:recurrence}
For every nonzero normalized state $H$ and any occupied pivot class $(a,r)$,
\begin{equation}
 \boxed{\displaystyle
 \Phi(H)=\sum_{u=0}^{r} I_a(u)
 \sum_{\nu\in\mathcal{A}(H,a,r,u)}
 \left(\prod_{b,s}W_{b,s}(\nu)\right)\Phi(H^\nu).}
 \label{eq:main-recurrence}
\end{equation}
Together with $\Phi(0)=1$, this recurrence counts every invariant completion exactly once.
\end{theorem}
\begin{proof}
Fix one labeled pivot cycle in the chosen class. Every invariant completion determines a unique subset of its internal edge orbits and a unique subset of the joining orbits between the pivot and each other active cycle. The internal subset determines $u$, while the joining subsets determine the orbit count $k$ for every target cycle and hence determine the allocation $\nu$.

For fixed $u$ and $\nu$, there are $I_a(u)$ internal choices. Within each target class, the number of assignments of the orbit counts to labeled cycles is the multinomial coefficient in \cref{eq:class-weight}. For every assigned target, \cref{lem:cross} gives $\binom{g_{ab}}{k}$ choices of joining orbits. These choices involve disjoint sets of possible edges, so their multiplicities multiply to the coefficient in \cref{eq:main-recurrence}.

Once those choices are fixed, all remaining undecided edges lie entirely among the nonexhausted target cycles. Their residual degrees are exactly those recorded by $H^\nu$, and \cref{thm:sufficient} gives $\Phi(H^\nu)$ remaining choices. Conversely, any local choice counted by the coefficient and any completion counted by $\Phi(H^\nu)$ combine into one invariant completion of $H$. The local choices can be recovered from that completion, so the construction neither loses nor duplicates an edge set. Finally, $N(H^\nu)\leq N(H)-a$, which makes induction from the empty state valid.
\end{proof}

There is no factor $H_{a,r}$ for selecting the pivot. The pivot is a fixed representative used to decompose a labeled counting problem, not an additional part of the output being counted. Multiplying by the number of available pivots would count the same edge set repeatedly; the labeled multiplicities that are required occur in the choices of target cycles in \cref{eq:class-weight}.

For $d=4$, the internal contribution $u$ takes only the values $0,1,2,3,4$, and each positive joining choice consumes at least one unit of the pivot's residual degree. Consequently, at most four target cycles can receive joining edges in any transition. This restriction concerns the number of affected cycles, not the number of original edges, since a selected joining orbit can contain many edges.

\subsection{Evaluation and admissibility checks}

A memoized evaluation stores one integer for each normalized histogram and reuses it whenever the same remaining problem occurs. The recurrence uses only exact integer additions and multiplications, while the division in Burnside's formula is postponed until all class contributions have been accumulated. A high-level evaluation is given in \cref{alg:count}; no graph representatives or isomorphism comparisons are required inside this calculation.

\begin{algorithm}[tb]
\caption{Fixed-point count for a normalized cycle histogram}
\label{alg:count}
\begin{algorithmic}[1]
\Function{Count}{$H$}
  \State Remove residual-degree-zero cycles and canonicalize $H$
  \If{$H=0$} \State \Return $1$ \EndIf
  \If{$H$ has a stored value} \State \Return the stored value \EndIf
  \If{$R(H)$ is odd or some occupied $(a,r)$ has $r>N(H)-1$}
    \State Store and \Return $0$
  \EndIf
  \State Choose an occupied pivot class $(a,r)$ and set $V\gets0$
  \For{$u=0,\ldots,r$ with $I_a(u)>0$}
    \For{$\nu\in\mathcal{A}(H,a,r,u)$}
      \State $V\gets V+I_a(u)\bigl(\prod_{b,s}W_{b,s}(\nu)\bigr)
      \Call{Count}{H^\nu}$
    \EndFor
  \EndFor
  \State Store $V$ at $H$ and \Return $V$
\EndFunction
\end{algorithmic}
\end{algorithm}

Two inexpensive necessary conditions follow directly from the meaning of a remaining state. The total residual degree $R(H)$ must be even, and each residual degree must be at most $N(H)-1$, because the future graph is simple and is supported only on the active vertices. The factors $a$ in $R(H)$ are essential. The unweighted sum $\sum_{a,r}rH_{a,r}$ counts demand per cycle rather than per vertex and is not a valid parity criterion.

A stronger local rejection rule is available before allocations are expanded. After removing a proposed pivot, form the polynomial
\begin{equation}
 Q_{H,a,r}(x)=P_a(x)
 \prod_{b,s}\left(
 \sum_{k=0}^{\min\{g_{ab},\lfloor s/\beta_{ab}\rfloor\}}
 \binom{g_{ab}}{k}x^{k\alpha_{ab}}
 \right)^{h_{b,s}},
 \label{eq:local-feasibility}
\end{equation}
truncated after degree $r$. If $[x^r]Q_{H,a,r}(x)=0$, then the pivot cannot be completed and $\Phi(H)=0$. A positive coefficient is only a necessary local feasibility condition, since the remaining targets may still fail to complete one another. This optional rejection rule does not change the recurrence or its proof.

\section{State and operation bounds}
\label{sec:complexity}

The number of completed graphs, the number of stored histograms, and the number of local allocation steps measure different aspects of the calculation. A small cache does not by itself imply a correspondingly small total operation count, because several allocations may lead to the same stored state. The following bounds account separately for the possible histograms and for the number of aggregated transitions leaving one histogram.

For a fixed initial partition $\lambda=(1^{m_1}2^{m_2}\cdots)$, each cycle of length $a$ is either absent from the active state or belongs to one of the $d$ positive residual-degree classes. Therefore the number of potentially stored states for that partition is at most
\begin{equation}
 S_d(\lambda)\leq\prod_{a:m_a>0}\binom{m_a+d}{d}.
 \label{eq:fixed-partition-bound}
\end{equation}
The bound ignores feasibility restrictions and correlations between residual classes. For the identity permutation it reduces to $\binom{n+d}{d}$, which is polynomial in $n$ when $d$ is fixed; thus the identity term should not be assumed to be the most difficult fixed-point calculation merely because it fixes the largest number of graphs.

\begin{proposition}[Global state bound]
\label{prop:states}
Let $\mathcal{H}_d(n)$ be the set of normalized histograms with $N(H)\leq n$, for a fixed integer $d\geq1$. Then
\begin{equation}
 |\mathcal{H}_d(n)|=
 \sum_{m=0}^{n}[x^m]\prod_{a\geq1}(1-x^a)^{-d}
 \leq \exp\left(\pi\sqrt{\frac{2dn}{3}}\right)
 \label{eq:global-state-bound}
\end{equation}
for $n\geq1$.
\end{proposition}
\begin{proof}
A cycle of length $a$ can occur with any nonnegative multiplicity in each of the $d$ positive residual-degree classes, so its contribution to the histogram generating function is $(1-x^a)^{-d}$. Multiplying over lengths and summing the coefficients through degree $n$ gives the equality in \cref{eq:global-state-bound}.

For $t>0$, nonnegativity of the coefficients gives
\begin{align}
 |\mathcal{H}_d(n)|
 &\leq e^{tn}\prod_{a\geq1}(1-e^{-ta})^{-d},\\
 \log\prod_{a\geq1}(1-e^{-ta})^{-d}
 &=d\sum_{k\geq1}\frac{1}{k(e^{tk}-1)}
 \leq\frac{d}{t}\sum_{k\geq1}\frac{1}{k^2}
 =\frac{d\pi^2}{6t}.
\end{align}
Using $e^{tk}-1\geq tk$ and then setting $t=\pi\sqrt{d/(6n)}$ proves the stated inequality.
\end{proof}

\begin{lemma}[Transitions from one state]
\label{lem:transitions}
For fixed $d$, a state with at most $n$ active vertices has $O_d(n^d)$ aggregated transitions in \cref{eq:main-recurrence}. Those transitions and their integer weights can be generated with polynomial overhead in $n$ per transition.
\end{lemma}
\begin{proof}
There are at most $n$ occupied target classes, and $k\leq d$ for every positive joining choice because $k\alpha_{ab}\leq r\leq d$. Thus the allocation has at most $dn$ possible positive coordinates $\nu_{b,s,k}$ with $k\geq1$. Equation~\eqref{eq:pivot-budget} implies
\begin{equation}
 \sum_{b,s}\sum_{k\geq1}\nu_{b,s,k}\leq d.
\end{equation}
The number of nonnegative vectors on at most $dn$ coordinates with total at most $d$ is at most $\binom{dn+d}{d}=O_d(n^d)$. Once the positive coordinates are fixed, $\nu_{b,s,0}$ is determined by \cref{eq:class-conservation}, and there are at most $d+1$ internal-degree choices. Enumerating these bounded-support vectors and checking the class capacities and degree equation gives the bound, while the updated histogram and combinatorial weights require only polynomial work.
\end{proof}

\begin{theorem}[Fixed-degree complexity]
\label{thm:complexity}
For each fixed $d\geq1$, \cref{eq:main-recurrence,eq:integer-burnside} give an exact algorithm for $U_d(n)$ with running time $\exp(O_d(\sqrt n))$ and storage $\exp(O_d(\sqrt n))$, measured in bit operations and bits respectively. With a common cache over all partitions, the number of stored fixed-point values is bounded by \cref{eq:global-state-bound}.
\end{theorem}
\begin{proof}
Every state reached from a partition of $n$ lies in $\mathcal{H}_d(n)$, and \cref{thm:sufficient} permits the same cached value to be reused across initial partitions. By \cref{lem:transitions}, the total number of arithmetic operations is bounded by a polynomial in $n$ times $|\mathcal{H}_d(n)|$, apart from partition generation and the final accumulation. Ordinary partitions are the one-color case of \cref{prop:states}, so there are also at most $\exp(O(\sqrt n))$ initial partitions to process.

For any remaining state, a completion has at most $dn/2$ edges. Its count is bounded by
\begin{equation}
 \sum_{e=0}^{\min\{\lfloor dn/2\rfloor,\binom n2\}}
 \binom{\binom n2}{e},
 \label{eq:count-bit-bound}
\end{equation}
whose bit length is $O_d(n\log(n+1))$. Local weights have polynomial bit length as well, and the class-size factors $n!/z_\lambda$ and the final integer sum add only $O(n\log(n+1))$ bits. Integer additions, multiplications, and exact divisions therefore incur polynomial overhead, which preserves the $\exp(O_d(\sqrt n))$ bound. A state key has polynomial length, so the same argument applies to storage.
\end{proof}

Separate caches for different partition tasks can reduce communication and permit independent scheduling, but they may evaluate the same histogram more than once. Multiplying the preceding bound by the number of partitions still gives $\exp(O_d(\sqrt n))$ time, although it changes the constant in the exponent and the practical amount of reuse. The bound is in the vertex count $n$ for fixed $d$; it does not imply polynomial time in $n$, and it does not provide a uniform fixed-degree estimate when $d$ grows with $n$.

\section{Recovering connected counts}
\label{sec:connected}

Let $C_d(n)$ count connected unlabeled $d$-regular graphs for $n\geq1$. Every graph counted by $U_d(n)$ is a unique multiset of connected graphs of the same degree, so the ordinary generating functions satisfy the Euler product \citep{HararyPalmer1973,OEISA033301,OEISA006820}
\begin{equation}
 A_d(x)=\sum_{n\geq0}U_d(n)x^n
       =\prod_{j\geq1}(1-x^j)^{-C_d(j)}.
 \label{eq:euler}
\end{equation}
This identity avoids storing component information during the fixed-point calculation, where such information would generally invalidate the histogram sufficiency argument.

Define $b_d(n)$ by $xA_d'(x)/A_d(x)=\sum_{n\geq1}b_d(n)x^n$. Taking the logarithmic derivative of \cref{eq:euler} gives the two triangular recurrences
\begin{align}
 b_d(n)&=nU_d(n)-\sum_{j=1}^{n-1}b_d(j)U_d(n-j),
 \label{eq:b-recurrence}\\
 C_d(n)&=\frac{1}{n}\left(b_d(n)-
 \sum_{\substack{j\mid n\\j<n}}jC_d(j)\right).
 \label{eq:connected-recurrence}
\end{align}
All divisions in the second recurrence must be exact for valid input counts. Reconstructing $U_d(n)$ from the resulting connected counts is a useful arithmetic check, but it is not an independent verification of the original fixed-point calculation.

A nonempty connected $d$-regular simple graph has at least $d+1$ vertices. Hence, for $n\geq2(d+1)$, the disconnected count can also be obtained from previously known connected counts as
\begin{equation}
 U_d(n)-C_d(n)
 =[x^n]\prod_{j=d+1}^{n-d-1}(1-x^j)^{-C_d(j)}.
 \label{eq:disconnected-short}
\end{equation}
In particular, the disconnected quartic count at order 29 depends only on connected counts through order 24. The connected sequence here is indexed from $n=1$; the special value assigned to the empty graph in OEIS A006820 is not used as a component in \cref{eq:euler}.

\section{Computational verification and measurements}
\label{sec:experiments}

\subsection{Independent small-instance verification}

The algebraic recurrence was checked against a separately implemented edge-orbit calculation that retains one degree coordinate for every labeled vertex. In that calculation, the permutation acts explicitly on unordered vertex pairs, and each orbit is generated by repeated application of the permutation rather than by the greatest-common-divisor formulas. If $\delta_v(O)$ is the number of edges of orbit $O$ incident with vertex $v$, the independent count for target degrees $\rho_v$ is
\begin{equation}
 [\prod_v x_v^{\rho_v}]
 \prod_{O}\left(1+\prod_v x_v^{\delta_v(O)}\right).
 \label{eq:vertex-verifier}
\end{equation}
The product is evaluated with exact coefficients, discarding a monomial as soon as it exceeds a target degree in any coordinate.

The comparison covered every partition of $n$ for $0\leq n\leq7$ and every degree $0\leq d\leq4$, giving 225 regular fixed-point comparisons. To check the remaining-state recurrence rather than only regular starting states, every partition with $1\leq n\leq5$ was also evaluated with every assignment of residual degrees in $\{0,1,2\}$ to its cycles. These 582 additional comparisons include exhausted cycles, unequal demands, and infeasible states. All 807 comparisons agreed exactly, as summarized in \cref{tab:verification}.

\begin{table}[tb]
\centering
\caption{Completed checks of the recurrence in this manuscript. The first two rows compare different counting formulations. The last row checks an algebraic identity within the aggregated formulation.}
\label{tab:verification}
\small
\begin{tabularx}{\textwidth}{@{}Xrr@{}}
\toprule
Comparison and input range & Cases & Mismatches\\
\midrule
Vertex-orbit product, all partitions of $0\leq n\leq7$, $0\leq d\leq4$ & 225 & 0\\
Vertex-orbit product, all partitions of $1\leq n\leq5$, all cycle demands in $\{0,1,2\}$ & 582 & 0\\
Complement identity, all partitions of $1\leq n\leq9$, $0\leq d<n$ & 686 & 0\\
\bottomrule
\end{tabularx}
\end{table}

Complementation gives a further exact identity for every permutation type,
\begin{equation}
 F_d(\lambda)=F_{n-1-d}(\lambda),
 \label{eq:complement}
\end{equation}
since complementation commutes with vertex relabeling. A total of 686 checks of this identity were completed for the ranges in \cref{tab:verification}. These checks and the independent comparisons concern a reference implementation of the displayed recurrence; they do not, by themselves, certify that a separately maintained large-order implementation follows the same recurrence in every detail.

\subsection{Quartic counts through order 50}

The completed calculation supplies 22 rows for $29\leq n\leq50$, with unrestricted counts $U_4(n)$, connected counts $C_4(n)$, disconnected counts $D_4(n)=U_4(n)-C_4(n)$, and labeled counts $L_4(n)=F_4(1^n)$. These quantities correspond respectively to OEIS A033301, A006820, A033483, and A005815 \citep{OEISA033301,OEISA006820,OEISA033483,OEISA005815}. The connected values are obtained from the unrestricted values by \cref{eq:b-recurrence,eq:connected-recurrence}; neither the connected nor the disconnected column represents a separate large-order fixed-point computation.

At the reference snapshot of \ReferenceSnapshot, the unrestricted and connected OEIS tables end at order 28, whereas the disconnected table ends at order 33 and the labeled table extends through order 260. The present unrestricted and connected tables therefore give 22 additional orders beyond those two reference tables, while the disconnected table gives 17 additional orders beyond its reference table. The labeled values are existing reference values used for comparison rather than a new sequence extension, as summarized in \cref{tab:coverage}.

\begin{table}[tb]
\centering
\caption{Coverage of the supplied quartic results and the reference tables at \ReferenceSnapshot. Every computed column in this manuscript covers $n=29,\ldots,50$. The final column refers to additional orders beyond the cited reference snapshot, not to OEIS acceptance or to an independent recomputation of the unrestricted totals.}
\label{tab:coverage}
\small
\begin{tabularx}{\textwidth}{@{}llrX@{}}
\toprule
Quantity & OEIS entry & Reference endpoint & Role of the supplied values\\
\midrule
$U_4(n)$ & A033301 & 28 & 22 additional orders, $29$--$50$\\
$C_4(n)$ & A006820 & 28 & 22 additional orders, $29$--$50$\\
$D_4(n)$ & A033483 & 33 & Agreement at $29$--$33$; 17 additional orders, $34$--$50$\\
$L_4(n)$ & A005815 & 260 & Agreement at all 22 orders, $29$--$50$\\
\bottomrule
\end{tabularx}
\end{table}

All 88 supplied integers are reproduced in \cref{tab:unrestricted,tab:connected,tab:disconnected,tab:labeled}. At the largest completed order, the unrestricted and connected totals each have 52 decimal digits and are
\begin{align}
 U_4(50)&=\TotalAtFifty,\\
 C_4(50)&=\ConnectedAtFifty,
\end{align}
with disconnected difference
\begin{equation}
 D_4(50)=\DisconnectedAtFifty.
\end{equation}
The largest labeled value has 117 digits and is printed in full in \cref{tab:labeled}. Decimal line breaks in that table do not truncate or round the integers.

\subsection{Arithmetic reconciliation and reference comparisons}

The supplied summary, the two comma-separated tables, and the four indexed sequence files were reconciled using integer arithmetic, and all eight supplied SHA-256 checksums matched. The full table contains each order from 29 to 50 exactly once, and all 88 numerical entries agree across its alternative representations. 

The lower unrestricted coefficients for $0\leq n\leq28$ were taken from the cited OEIS snapshot and combined with the supplied values at $29\leq n\leq50$. Reapplying the inverse Euler transform reproduced every supplied connected value, with exact division at every positive order through 50. A separate forward calculation formed the truncated product in \cref{eq:euler} using the factors
\begin{equation}
 (1-x^j)^{-C_4(j)}
 =\sum_{k\geq0}\binom{C_4(j)+k-1}{k}x^{jk}
\end{equation}
for positive $C_4(j)$, while zero exponents contributed one. This product recovered all 51 unrestricted coefficients at orders 0--50, including the empty-graph coefficient. The empty graph was not treated as a connected component.

All 22 supplied labeled values agree exactly with the corresponding entries of the published A005815 table. They also agree with a separate exact-rational power-series evaluation of the differential equation for the labeled exponential generating function given in that entry \citep{OEISA005815}. The five supplied disconnected values at orders 29--33 agree with A033483 \citep{OEISA033483}. Since each nonempty quartic component has at least five vertices, these five disconnected values depend only on connected counts through order 28 and provide a check based on previously available lower-order data.

The large-order material supplied for this revision contains final totals and identity contributions, but not the full collection of records $(\lambda,z_\lambda,F_4(\lambda))$. Consequently, partition completeness, per-partition outputs, and the divisibility of the original accumulated Burnside numerators were not rechecked in this reconciliation. The Euler relations and the identity comparisons do not certify the nonidentity contributions, and multiplying a supplied total by $n!$ would not constitute a check of the original numerator.

An independent large-order audit can retain every partition record and compare selected nonidentity terms using \cref{eq:vertex-verifier}, with the selection covering fixed vertices, repeated short cycles, even cycles, and unequal cycle lengths. The numerical extension reported here is a computed result accompanied by the checks in \cref{tab:verification}, rather than a claim that the extended unrestricted totals have already been independently verified or incorporated into the reference database.

\subsection{Preliminary resource measurements}

The available resource summary concerns an earlier preliminary run at $n=29$, which used 4,565 partition tasks with up to 56 allocated CPU cores; no per-order resource measurements for the completed extension through $n=50$ accompany the supplied result tables. Its reported timing and memory summary is reproduced in \cref{tab:preliminary}, separately from the small-instance reference calculation above. The processor model, software versions, repetition protocol, and correspondence between the implemented pivot and pruning rules and those described here still require completion before these figures are used for a comparative performance claim.

\begin{table}[tb]
\centering
\caption{Reported preliminary resource summary for the quartic implementation at $n=29$. These values come from the supplied run summary and were not remeasured as part of the independent small-instance verification.}
\label{tab:preliminary}
\begin{tabularx}{\textwidth}{@{}Xr@{}}
\toprule
Quantity & Recorded value\\
\midrule
Partition tasks & 4,565\\
Allocated CPU cores & 56\\
Sum of GP counting CPU times & 272.984 s\\
Elapsed time from first task start to last task finish & 106 s\\
Largest individual partition counting CPU time & 2.511 s\\
Largest cache size in one partition task & 7,227 states\\
Largest reported single-process resident set size & 11,164 KiB\\
\bottomrule
\end{tabularx}
\end{table}

The summed counting CPU time and the task-span elapsed time have different definitions. The latter can include process startup, scheduling gaps, and intervals not covered by the counting timer, so their ratio is not a parallel speedup. Likewise, the largest partition cache records stored remaining states, not all allocations examined while those states were evaluated. A completed report should distinguish cache misses, cache hits, completed transitions, and internal allocation steps.

\subsection{Matched comparisons and scaling measurements}

A controlled comparison of state representations can retain the same edge-orbit formulas, pivot rule, and degree bounds while changing only the stored remaining state. One variant retains the identity of every permutation cycle and its residual degree, a second memoizes those identity-retaining states, and a third uses the histogram recurrence in \cref{eq:main-recurrence}. Recording both state counts and allocation work distinguishes the benefit of merging states from the benefit of reducing local branching.

\pending{Matched measurements for the three state representations, per-order resource measurements through $n=50$, and any additional degree values remain unavailable. The completed integer tables do not supply a timing ratio or a scaling curve.}

The public labeled-degree program should be compared with the identity-permutation contribution only, using the same degree, arithmetic requirements, and hardware. An explicit graph generator can instead be compared on its count-only mode over a common feasible range, with connected and unrestricted outputs reconciled before comparison. These two comparisons answer different questions, and the fact that a generator internally constructs representatives should remain explicit even when graph serialization is disabled.

For each completed order, a scaling study can report the sum of partition CPU times, total completed transitions, maximum and total partition-cache sizes, and peak resident memory under a fixed cache policy. The distribution over cycle types is as relevant as the total, because partitions with many equal short cycles need not behave like partitions with a few long cycles. Runs at degrees three and five would examine how local branching changes with degree, but such runs should be reported only after their own outputs and resource records have been checked.

\section{Discussion}
\label{sec:discussion}

The recurrence separates the combinatorial multiplicity of local edge choices from the number of distinct remaining counting problems. Large binomial and multinomial coefficients account for many labeled choices without introducing a separate recursive branch for every choice. The stored value is nevertheless an exact count, because all choices represented by one weighted transition have the same remaining completion count under \cref{thm:sufficient}.

The same interpretation explains both the usefulness and the limits of the compression. It is sufficient to store cycle length and residual degree only because all undecided edges among active cycles remain available and the desired property is the degree condition. Direct restrictions on girth, diameter, or connectivity can depend on previously chosen paths and components, so they generally require additional state information. Recovering connected regular-graph counts through \cref{eq:euler} avoids this difficulty without claiming that the histogram itself determines connectivity.

The fixed-degree complexity bound concerns the explicitly stated recurrence rather than the provenance of a particular implementation. The supplied counts through order 50 show the completed numerical range of the present calculation, but the 117-digit identity contributions and the 52-digit unrestricted totals do not measure its operation count. A speed comparison with other exact counting methods still requires matched measurements, including the internal allocation work that is not visible in a peak cache-size measurement.

\section{Conclusion}

Complete-cycle elimination turns each Burnside fixed-point calculation into a recurrence on the numbers of remaining cycles of each length and residual degree. Internal distance orbits and joining orbits give explicit local coefficients, while multinomial weights preserve the identities of target choices that disappear from the stored state. The resulting recurrence counts invariant labeled edge sets exactly and gives unrestricted unlabeled counts after the conjugacy-class average.

For fixed degree, the number of normalized histograms is bounded by a colored-partition generating function, and the residual-degree budget bounds the number of affected target cycles in each transition. These observations yield an $\exp(O_d(\sqrt n))$ time and storage bound and allow connected counts to be recovered afterward by an inverse Euler transform. The quartic calculation supplies unrestricted and connected counts through order 50, together with the disconnected differences and labeled identity contributions. Small-instance fixed-point comparisons, agreement of all 22 supplied identity terms with published labeled counts, and exact Euler reconciliation provide distinct checks, while an independent audit of the large-order nonidentity terms and matched resource measurements remain separate tasks.

\section*{Acknowledgements}

The author thanks the High-Performance Computing Center, School of Mathematics and Statistics, Nantong University, for providing the computational resources used in this work.

\section*{Data and code availability}

The code and data supporting this study are available on Zenodo at \url{https://doi.org/10.5281/zenodo.22961966}. The deposited archive contains the PARI/GP implementation of the permutation-fixed counting recurrence, Python and shell scripts for task preparation, parallel execution, aggregation, and exact-integer validation, together with the permutation-fixed count records and derived quartic graph sequences for orders \(n=27\)--\(50\). The sequence data include counts of all unlabeled, connected unlabeled, disconnected unlabeled, and labeled simple 4-regular graphs. All numerical values are stored as exact integers. The archive contains counting data and source code but does not include graph representatives.

\appendix
\section{Small examples and boundary cases}
\label{app:examples}

For a permutation consisting of one $a$-cycle, no joining choices remain and the fixed-point count is simply $[x^d]P_a(x)$. In the quartic case,
\begin{equation}
 F_4((a))=\binom{\lfloor(a-1)/2\rfloor}{2},
 \label{eq:single-cycle-quartic}
\end{equation}
where a binomial coefficient is zero when its lower index exceeds its upper index. For even $a$, the matching orbit cannot be included in a degree-four internal choice, since the other internal orbits all contribute even degree. This gives a direct check of the parity and matching conventions in \cref{eq:internal}.

For two distinct cycles of lengths four and six, there are two joining orbits, each contributing degree three to the 4-cycle and degree two to the 6-cycle. Their factor in \cref{eq:cycle-product} is $(1+x^3y^2)^2$. A degree-four budget on the 4-cycle permits at most one of these orbits, with multiplicity two. Reversing the pivot changes which degree increment is charged to the pivot but does not change the underlying fixed-point count.

To see the need for both factors in \cref{eq:class-weight}, consider eight labeled target cycles of length three and a pivot of the same length. An allocation that joins exactly three targets with one orbit each has target-selection multiplicity $\binom83=56$. There are three joining orbits for each selected target, so the complete local factor is
\begin{equation}
 \binom83\binom31^3=1512.
\end{equation}
The target-selection factor alone would omit most labeled edge sets, even though all these choices may lead to the same remaining histogram.

Finally, a state consisting of two fixed vertices each of residual degree one has exactly one completion. Selecting one as the pivot leaves a single possible joining edge, and the recurrence returns one without multiplying by the two available pivot identities. Exhausted cycles have the complementary boundary behavior, because all their future incident edges are forced absent and deleting them also contributes a factor of one.

\section{Complete quartic count tables}
\label{app:tables}

The following tables reproduce all four supplied sequences at orders 29--50 without rounding. The unrestricted and connected counts are listed separately so that each complete integer fits on one line, and the disconnected column is printed independently to retain the full supplied data.

For the labeled sequence, an integer longer than 60 digits is continued on the following line within the same table row. Concatenating those lines recovers the integer exactly; the line break has no numerical meaning. The supplied labeled values agree with the already published A005815 values and are not claimed as new labeled counts.

\begingroup
\small
\renewcommand{\arraystretch}{1.0}
\setlength{\LTleft}{0pt}
\setlength{\LTright}{\fill}
\setlength{\LTcapwidth}{\textwidth}
\setlength{\tabcolsep}{8pt}
\begin{longtable}{@{}>{\raggedleft\arraybackslash}p{0.7cm} p{\dimexpr\textwidth-0.7cm-16pt\relax}@{}}
\caption{Computed unrestricted unlabeled quartic graph counts for $29\leq n\leq50$. These orders are beyond the A033301 snapshot through order 28.}\label{tab:unrestricted}\\
\toprule
$n$ & $U_4(n)$\\
\midrule
\endfirsthead
\multicolumn{2}{@{}l}{\textit{Table \thetable\ continued}}\\
\toprule
$n$ & $U_4(n)$\\
\midrule
\endhead
\midrule
\multicolumn{2}{r@{}}{\textit{Continued on next page}}\\
\endfoot
\bottomrule
\endlastfoot
\UnrestrictedTableRows
\end{longtable}
\endgroup

\clearpage
\begingroup
\small
\renewcommand{\arraystretch}{1.0}
\setlength{\LTleft}{0pt}
\setlength{\LTright}{\fill}
\setlength{\LTcapwidth}{\textwidth}
\setlength{\tabcolsep}{8pt}
\begin{longtable}{@{}>{\raggedleft\arraybackslash}p{0.7cm} p{\dimexpr\textwidth-0.7cm-16pt\relax}@{}}
\caption{Connected unlabeled quartic graph counts for $29\leq n\leq50$, obtained from the unrestricted counts by the inverse Euler transform. These orders are beyond the A006820 snapshot through order 28.}\label{tab:connected}\\
\toprule
$n$ & $C_4(n)$\\
\midrule
\endfirsthead
\multicolumn{2}{@{}l}{\textit{Table \thetable\ continued}}\\
\toprule
$n$ & $C_4(n)$\\
\midrule
\endhead
\midrule
\multicolumn{2}{r@{}}{\textit{Continued on next page}}\\
\endfoot
\bottomrule
\endlastfoot
\ConnectedTableRows
\end{longtable}
\endgroup

\clearpage
\begingroup
\small
\renewcommand{\arraystretch}{1.0}
\setlength{\LTleft}{0pt}
\setlength{\LTright}{\fill}
\setlength{\LTcapwidth}{\textwidth}
\setlength{\tabcolsep}{8pt}
\begin{longtable}{@{}>{\raggedleft\arraybackslash}p{0.7cm} p{\dimexpr\textwidth-0.7cm-16pt\relax}@{}}
\caption{Disconnected unlabeled quartic graph counts for $29\leq n\leq50$. The values at orders 29--33 agree with the A033483 snapshot, and orders 34--50 are beyond that snapshot.}\label{tab:disconnected}\\
\toprule
$n$ & $D_4(n)$\\
\midrule
\endfirsthead
\multicolumn{2}{@{}l}{\textit{Table \thetable\ continued}}\\
\toprule
$n$ & $D_4(n)$\\
\midrule
\endhead
\midrule
\multicolumn{2}{r@{}}{\textit{Continued on next page}}\\
\endfoot
\bottomrule
\endlastfoot
\DisconnectedTableRows
\end{longtable}
\endgroup

\clearpage
\begingroup
\small
\renewcommand{\arraystretch}{1.0}
\setlength{\LTleft}{0pt}
\setlength{\LTright}{\fill}
\setlength{\LTcapwidth}{\textwidth}
\setlength{\tabcolsep}{8pt}
\begin{longtable}{@{}>{\raggedleft\arraybackslash}p{0.7cm} p{\dimexpr\textwidth-0.7cm-16pt\relax}@{}}
\caption{Labeled quartic graph counts for $29\leq n\leq50$, equal to the identity-permutation contributions. All entries agree with A005815. Continued digit strings within a row are concatenated without a separator.}\label{tab:labeled}\\
\toprule
$n$ & $L_4(n)=F_4(1^n)$\\
\midrule
\endfirsthead
\multicolumn{2}{@{}l}{\textit{Table \thetable\ continued}}\\
\toprule
$n$ & $L_4(n)=F_4(1^n)$\\
\midrule
\endhead
\midrule
\multicolumn{2}{r@{}}{\textit{Continued on next page}}\\
\endfoot
\bottomrule
\endlastfoot
\LabeledTableRows
\end{longtable}
\endgroup

\end{document}